\documentclass[a4paper, english, 12pt, twoside]{amsart}

\usepackage{babel}
\usepackage[T1]{fontenc}
\usepackage{amssymb}
\usepackage{amsmath, amsfonts, mathrsfs, amscd}
\usepackage{xcolor}
\usepackage{eucal}
\usepackage{calligra}
\usepackage{etoolbox}

\newtheorem{theorem}{Theorem}[section]
\newtheorem{lemma}[theorem]{Lemma}

\newtheorem{corollary}[theorem]{Corollary}

\theoremstyle{definition}

\newtheorem{question}[theorem]{Question}
\newtheorem{remark}[theorem]{Remark}
\newtheorem{example}[theorem]{Example}

\def\pf{\begin{proof}}
\def\epf{\end{proof}}

\newcommand{\Na}{\mathbb{N}}
\newcommand{\kk}{\Bbbk}
\newcommand{\Z}{\mathbb{Z}}

\newcommand{\End}{\operatorname{End}}

\definecolor{rojo}{rgb}{1,0,0}

\begin{document}

\title[Hopf actions on central division algebras]
{Finite dimensional Hopf actions on central division algebras}

\author{Juan Cuadra}
\address{Department of Mathematics, University of Almer\'{\i}a, E-04120 Almer\'{\i}a,
Spain}
\email{jcdiaz@ual.es}

\author{Pavel Etingof}
\address{Department of Mathematics, Massachusetts Institute of Technology,
Cambridge, MA 02139, USA} \email{etingof@math.mit.edu}

\date{\today}

\begin{abstract}
Let $\kk$ be an algebraically closed field of characteristic zero. Let $D$ be a division algebra of degree $d$ over its center $Z(D)$. Assume that $\kk \subset Z(D)$. We show that a finite group $G$ faithfully grades $D$ if and only if $G$ contains a normal abelian subgroup of index dividing $d$. We also prove that if a finite dimensional Hopf algebra coacts on $D$ defining a Hopf-Galois extension, then its PI degree is at most $d^2$. Finally, we construct Hopf-Galois actions on division algebras of twisted group algebras attached to bijective cocycles.
\end{abstract}

\maketitle

\section{Introduction}

Understanding finite dimensional Hopf actions on rings (i.e., their finite quantum symmetries) is one of the central subjects in noncommutative geometry, which has been studied by a number of researchers for many years; see \cite{Mo}. While the problem of classification of such actions is intractable, the situation appears much more manageable for actions on {\it division algebras}. In addition, by Theorem 2.2 of \cite{SV}, any Hopf action on an integral domain which has a division algebra of quotients extends to the latter. This makes the problem of classification of finite dimensional Hopf actions on division algebras an important step towards understanding more general situations. \smallskip

We consider here this question when the ground field is algebraically closed of characteristic zero. The theory of central division algebras containing an algebraically closed field is quite rich and has deep connections to algebraic geometry. Striking examples in this setting are: the non-crossed product division algebra constructed by Amitsur in \cite{Am}; the division algebras realizing any finite group as the Galois group of a maximal subfield constructed by Fein, Saltman and Schacher in \cite{FSS}, or more recently, the example given by Saltman and Rowen in \cite{RS} of two division algebras whose tensor product over the ground field is not a domain. From a Hopf algebra point of view, this setting is also interesting. For instance, it is not known whether an arbitrary finite dimensional Hopf algebra (or even a semisimple one) can act on a division algebra inner faithfully (i.e., without factoring through a quotient Hopf algebra). In fact, Artamonov makes an even stronger conjecture in \cite[Conjecture 0.1]{Ar}: any finite dimensional Hopf algebra can act inner faithfully on a quantum torus (and thereby on its division algebra of quotients). One may wonder if this holds even if the parameters of the quantum torus are required to be roots of unity (then the corresponding division algebra is {\it central}). On the other hand, it is shown in \cite{EW1} that any semisimple Hopf action on a {\it commutative} division algebra (a field) must factor through a group action. Though this is not so for nonsemisimple Hopf algebras, the class of them that can act on fields inner faithfully (called Galois-theoretical) is quite restricted; see  \cite{EW2}. \smallskip

This leads to the following question, on which we focus in this paper:

\begin{question}
Which finite dimensional Hopf algebras can act inner faithfully on a central division algebra? On a central division algebra of a given degree?
\end{question}

This question appears to be difficult, and the answer is unknown even for
semisimple Hopf algebras. In fact, it is nontrivial even for the particular case of
the function algebra of a finite group $G$. In this case, a Hopf action on a division algebra $D$ is just a $G$-grading on $D$, and an inner faithful action is just a faithful grading; that is, $D_g\ne 0$ for all $g\in G$. \smallskip

Our first main result, established in Section 2, is that $G$ serves as a faithful grading group of a central division algebra of degree $d$ if and only if $G$ has a normal abelian subgroup of index dividing $d$. The proof relies on some basic results on central division algebras, and a theorem of Fein, Saltman and Schacher \cite{FSS} stating that for any finite group $H$ there exists a crossed product algebra constructed from $H$ that is an $H$-graded division algebra. We must emphasize at this point that gradings on division algebras were previously investigated; see Remark \ref{rem}. The novelty in our setting is that we allow the center of the division algebra to be an infinite extension of the base field, and gradings are not necessarily linear over the center. \smallskip

To explain our other results, recall that an important class of Hopf actions is the class of Hopf-Galois actions, i.e., those defining a Hopf-Galois extension. Question 3.5 of \cite{CEW} asks if an inner faithful semisimple Hopf action on a (central) division algebra must be Hopf-Galois. Our second result, given in Section 3, answers it in the affirmative for finite group actions. This is a generalization of the standard result in classical Galois theory. We also give an example showing that if the ground field is not algebraically closed, then the answer is negative, even for fields. Hence, the result of \cite{EW1} fails: there is an inner faithful action of a non-cocommutative semisimple Hopf algebra on a field that is not Hopf-Galois. \smallskip

Our third result, proved in Section 4, is a partial positive answer to Question 5.9 of \cite{EW1}. This question asks if the dimension of a simple comodule for a semisimple Hopf algebra acting inner faithfully on a PI domain of degree $d$ must be at most $d^2$. We show that this is the case when the action is Hopf-Galois. \smallskip

Finally, Section 5 contains our fourth result. It states that if $J$ is a twist on a group algebra of a finite group $G$ coming from a bijective cocycle on $G$ (see \cite[Section 4]{EG}), then the corresponding twisted group algebra can act inner faithfully on a division algebra of degree $|G|$. This provides many examples of noncocommutative semisimple Hopf actions on central division algebras. \bigskip

{\bf Notation and conventions.} Throughout $\kk$ will stand for an algebraically closed field of characteristic zero, unless specified otherwise. Vector spaces, linear maps, and unadorned tensor products $\otimes$ are over $\kk$. Algebras considered are associative and unital and all actions are understood to be $\kk$-linear. We denote by $\mathcal{H}$ and $\mathcal{K}$ finite dimensional Hopf algebras over $\kk$. \smallskip

\section{Faithful gradings on division algebras}

\subsection{The main theorem}
Let $D$ be a division algebra of degree $d$ whose center contains $\kk$. We say that $D$ is faithfully graded by a finite group $G$ if there is a direct sum decomposition $D=\oplus_{g \in G} D_g$ as $\kk$-vector spaces such that $D_g \neq 0$ and $D_gD_h=D_{gh}$ for all $g,h \in G$. \smallskip

The goal of this section is to prove the following.

\begin{theorem}\label{mainthe}
Let $G$ be a finite group. Then $G$ grades faithfully a central division algebra of degree $d$
with center containing ${\kk}$ if and only if $G$ contains a normal abelian subgroup of index dividing $d$.
\end{theorem}

Note that the case $d=1$ is standard. In this case $D$ is a field and the condition $D_gD_h=D_hD_g$ implies $gh=hg$. Thus $G$ is abelian. The existence of fields faithfully graded by any finite abelian group follows easily from Galois theory. \smallskip

The proof of Theorem \ref{mainthe} is contained in the next two subsections. \smallskip

\begin{remark}\label{rem}
Graded division algebras are also studied in \cite[Sections A.I.4 and C.I.1]{NO} and \cite{BSZ}, but the setting of these works is different from ours. Namely, a graded division algebra is not required to be a division algebra in the usual sense (only homogeneous nonzero elements are required to be invertible), and in \cite{BSZ} only algebras which are finite dimensional over an algebraically closed field are considered.
\end{remark}

\subsection{Auxiliary results}

\begin{lemma} \label{ab}
If $D_1$ is central, then $G$ is abelian.
\end{lemma}

\pf
Let $g,h \in G$. Take $x \in D_g,y \in D_h$ nonzero. Then $u:=xyx^{-1}y^{-1}$ lies in $D_{ghg^{-1}h^{-1}}$ and has norm $1$. Denote the order of $ghg^{-1}h^{-1}$ by $n$. Then $a:=u^n \in D_1 \subset Z(D),$ where $Z(D)$ is the center of $D$. Apply norm to both sides: $1=N(u)^n=a^d$. Now, $u^{nd}=a^d=1.$ Hence $u$ is a root of the polynomial $t^{nd}-1$. Since all $nd$-th roots of unity are in $Z(D)$, we have $u=\omega 1 \in D_1$, for such a root of unity $\omega$. So $ghg^{-1}h^{-1}=1$ and $g$ and $h$ commute.
\epf

\begin{remark} The condition that $F:=D_1$ is central is equivalent to $D$ being the twisted group algebra $F^bG$ for some 2-cocycle $b: G\times G\to F^\times$ (with the trivial $G$-action on $F$). The assumption
that $\kk$ is algebraically closed is essential for Lemma \ref{ab}: otherwise some non-abelian groups can arise, see \cite{AHN}.
\end{remark}

\begin{lemma}\label{forwdir}
Let $D$ be a division
algebra of degree $d$ whose center contains $\kk$. If $D$ is faithfully graded by a finite group $G$, then $G$ has a normal abelian subgroup of index $r$ dividing $d$.
\end{lemma}

\pf
Put $Z=Z(D)$ and $Z_1=Z(D_1)$. Consider the field $ZZ_1$ inside $D$ and set $r=[ZZ_1:Z].$ We can identify the division algebra $ZD_1$ with $D_1 \otimes_{Z_1} (ZZ_1)$ under multiplication. Then the center of $ZD_1$ is $ZZ_1$ and we obtain that $D_1$ has finite degree, say $m$. \smallskip

We claim that the dimension of any maximal subfield $L$ of $ZD_1$ over $Z$ is $rm$. Indeed, we have:
$$\begin{array}{l}
[ZD_1:Z]=[ZD_1:L][L:Z]=m[L:Z], \vspace{5pt} \\
\lbrack ZD_1:Z \rbrack=[ZD_1:ZZ_1][ZZ_1:Z]=m^2r.
\end{array}$$
Since $L$ is contained in a maximal subfield of $D$, we get that $rm$ divides $d$. \smallskip

Write $Q=ZD_1$. Consider the centralizer $C(Q)$ of $Q$ in $D$. Then $C(Q)=C(D_1)$. It is not difficult to see that $C(Q)$ is a graded subalgebra of $D$, that is,
$C(Q)=\oplus_{g \in G} (C(Q) \cap D_g)$. Taking nonzero components, $C(Q)$ is faithfully graded by a subgroup $A$ of $G$. This subgroup must be normal because conjugation by an element in $D_g$ stabilizes $D_1$. We next show that $A$ is abelian and has index $r$. \smallskip

The homogeneous component of degree $1$ of $C(Q)$ is $C(Q) \cap D_1=Z_1.$ The center of $C(Q)$ equals $C(Q) \cap C(C(Q)) = C(Q) \cap Q = ZZ_1,$ where we used the double centralizer theorem. As $Z_1 \subseteq ZZ_1$, by Lemma \ref{ab}, $A$ is abelian. On the other hand, we have the equalities:
$$\begin{array}{l}
[D:Z_1] = [D:C(Q)][C(Q):Z_1]=[D:C(Q)] \vert A \vert \vspace{3pt} \\
\lbrack D:Z_1 \rbrack = [D:D_1][D_1:Z_1] = \vert G \vert m^2.
\end{array}$$
We now check that $[D:C(Q)]=rm^2$. Substituting in the previous equalities, we will obtain the statement. We have $[D:Z]=[Q:Z][C(Q):Z]$ by the double centralizer theorem. Then:
$$[D:Z] = [D:C(Q)][C(Q):Z] = [D:C(Q)][D:Z][Q:Z]^{-1}.$$
Hence $[D:C(Q)]=[Q:Z]=rm^2.$
\epf

We next show that any finite group with a normal abelian subgroup of index $r$ dividing $d$ grades faithfully a central division algebra of degree $d$.

\begin{lemma}\label{enlarge}
Let ${\mathcal{H}}$ be a finite dimensional Hopf algebra over $\kk$ acting inner faithfully on a central division algebra $D$ of degree $r$ (with $\kk \subset Z(D)$).
Then ${\mathcal{H}}$ acts inner faithfully on a central division algebra of degree $rn$ for any $n \in \Na$.
\end{lemma}

\pf
Let $q \in \kk$ be a root of unity of order $n$ and $\kk_q[x,y]$ the algebra generated by $x,y$ with relation $xy=qyx$. Set $D_q[x,y]:=D\otimes \kk_q[x,y]$. It is easy to see that $D_q[x,y]$ has no zero divisors. Its quotient division algebra $D_q(x,y)$ has center $Z(D)(x^n,y^n)$. Moreover, since $\kk_q(x,y)$ has degree $n$, $D_q(x,y)$ has degree $rn$. Letting $\mathcal{H}$ act on $x,y$ trivially, we endow $D_q(x,y)$ with an inner faithful action of $\mathcal{H}$.
\epf

Lemma \ref{enlarge} allows us to assume, without loss of generality, that $r=d$. \smallskip

Let $G$ be a finite group containing a normal abelian subgroup $A$ of index $d$.
Let $A^\vee$ be the character group of $A$, and
$m$ be the minimal number of generators of $A^\vee$ as an $H$-module.

\begin{lemma}\label{subgroup}
Set $H=G/A$, and let $n={\rm exp}(A)$ denote the exponent of $A$.
Then $G$ is a subgroup of ${\rm Fun}(H,\Z_n)^m \rtimes H$.
\end{lemma}

\pf  Let $f_1,\ldots,f_m$ be a minimal set of generators of $A^\vee$ as an $H$-module.
Consider the group algebra $\Z_n H$ and the $H$-module epimorphism
$$
\psi:(\Z_n H)^m \rightarrow A^\vee
$$
mapping $1 \in H$ in the $i$-th copy to $f_i$. By dualizing, we obtain a
monomorphism $\psi^\vee:A \rightarrow {\rm Fun}(H,\Z_n)^m$ (where we identify $\Z_n$ with the group of roots of unity of order $n$). Thus we can view $A$ as an $H$-submodule \smallskip of ${\rm Fun}(H,\Z_n)^m$.

Let $s:H \rightarrow G$ be a set-theoretical splitting. Then the map
$$
A \times H \rightarrow G, (a,h) \mapsto as(h)
$$
is bijective. The group law on $G$ defines a $2$-cocycle $c:H \times H \rightarrow A$, and we have $G=A_c\rtimes H$, the semidirect product twisted by $c$.
Consider now the $2$-cocycle $\psi^\vee(c) \in Z^2(H, {\rm Fun}(H,\Z_n)^m)$. We have an inclusion
$$
G=A_c\rtimes H\subset {\rm Fun}(H,\Z_n)^m_{\psi^\vee(c)}\rtimes H.
$$
But by the Shapiro Lemma,
$H^2(H,{\rm Fun}(H,\Z_n)^m)=0$, so $\psi^\vee(c)$ is a coboundary.
Hence, ${\rm Fun}(H,\Z_n)^m_{\psi^\vee(c)}\rtimes H\cong {\rm Fun}(H,\Z_n)^m\rtimes H$, and we have an inclusion
$G=A_c\rtimes H\subset {\rm Fun}(H,\Z_n)^m\rtimes H$ of $G$ into the usual (untwisted) semidirect product, as claimed.
\epf

Let $L$ be a field containing $\kk$, and assume that $L$ carries a faithful action of a finite group $H$. Consider the crossed product algebra $D=(L/L^H,H,b)$ for a $2$-cocycle $b:H \times H \rightarrow L^{\times}$. In other words, $D$ is the twisted group algebra $L_bH$ (which is an algebra over $L^H$).
Note that $D$ is naturally $H$-graded, with $D_h=Lh$. \smallskip

For a finite dimensional $\kk$-representation $V$ of $H$ let $D_t[V]$ be the tensor product $\kk[V]\otimes D$ with multiplication given by
$$
(f_1\otimes d_1)(f_2\otimes d_2)=f_1h(f_2)\otimes d_1d_2
$$
if $d_1\in D_h$; i.e., $D_t[V]$ is a twisted tensor product of $\kk[V]$ and $D$.

\begin{lemma}\label{division}
Suppose that $D$ is a division algebra. Then $D_t[V]$ is a domain, so the full localization $D_t(V)$ is a division algebra.
\end{lemma}

\begin{proof} We have a natural $H$-equivariant multiplication map
$$
\psi: L\otimes_{L^H}(L\otimes V)^H\to L\otimes V.
$$
We claim that $\psi$ is an isomorphism. Let $F$ be the algebraic closure of $L^H$, and $V_F:=F\otimes V$. Then, upon tensoring with $F$ over $L^H$, the map $\psi$ becomes the multiplication map ${\rm Fun}(H,F)\otimes_F {\rm Fun}(H,V_F)^H\to {\rm Fun}(H,V_F)$, which is an isomorphism because it is a surjective map between spaces of the same dimension. \smallskip

Regard $(L\otimes V)^H$ as a vector space over $L^H$ (of dimension $r:=\dim V$),
and let $x_1,\ldots,x_r$ be a basis of its dual space. Since $\psi$ is an isomorphism, we have an $H$-equivariant isomorphism
$L[V]=L[x_1,\ldots,x_r]$, where $hx_i=x_i$ for $h\in H$. Hence $D_t[V]=D[x_1,\ldots,x_r]$. This implies the desired statement by using the standard leading term argument.
\end{proof}

\subsection{Proof of Theorem \ref{mainthe}}
We first observe that by Lemmas \ref{enlarge} and \ref{subgroup}, it suffices to prove the theorem for the group ${\rm Fun}(H,\Z_n)^m \rtimes H$. Indeed, if ${\rm Fun}(H,\Z_n)^m \rtimes H$ faithfully grades a central division algebra $D$ of degree $d$, then $G$, which is a subgroup by Lemma \ref{subgroup}, faithfully grades some subalgebra $D'$ of $D.$ But $D'$ is a division algebra of degree dividing $d.$ So by Lemma \ref{enlarge}, $G$ faithfully grades some division algebra $D''$ of degree exactly $d.$ \par \smallskip

We realize $H$ as the Galois group of an extension of fields $L/K$ with $\kk \subset K$. Then $L$ admits a faithful action of $H$. By \cite[Corollary 5.5]{FSS}, $L/K$ can be chosen in such a way that there is a $2$-cocycle $b:H \times H \rightarrow L^{\times}$ such that the crossed product algebra $D:=(L/K,H,b)$ is a finite dimensional division algebra with center $K$. By definition, $D$ is faithfully graded by $H$. Consider the representation $V=(\kk H)^{\oplus m}$
with basis $x_{i,h},$ for $i=1,\ldots,m, h\in H,$ and action $h \cdot x_{i,h'}=x_{i,hh'}.$ By Lemma \ref{division}, $Q:=D_t(V)$ is a division algebra. Moreover, $Z(Q)=L(V)^G$, which implies that the degree of $Q$ is still $d$. We introduce a new grading on $Q$ by the group ${\rm Fun}(H,\Z_n)^m \rtimes H$ by setting $\deg(D_h)=h$ and
${\rm deg}(x_{i,h})=\delta_{i,h}$, the delta function of $h$ in the $i$-th copy of ${\rm Fun}(H,\Z_n)^m.$
This gives a division algebra with a faithful grading by $G$. Together with Lemma \ref{forwdir}, this implies the Theorem \ref{mainthe}. \qed \par \medskip

\begin{remark}\label{brownrem}
Here is another proof of the fact that any finite group $G$ can grade faithfully a central division algebra over $\kk$, not using \cite{FSS}. It was communicated to us by Ken Brown. \par \smallskip

We start with the following result in group theory: \par \smallskip

{\it \cite[Theorem 2]{Hi} Let $F$ be a free group and $N$ a normal subgroup of $F$. Then $F/[N,N]$ is torsion free.} \par \smallskip

Now pick a finite presentation of $G$, i.e., $G=F/N$, where $F$ is a finitely generated free group.
Let $K=F/[N,N]$. Then, $K$ is torsion-free. Also we have an exact sequence
$$
1\to N/[N,N]\to K\to G\to 1,
$$
i.e., $K$ is an extension of a finite group by a free abelian group of finite rank.
So by \cite[Corollary 2 and Section 3]{Br}, $A:=\kk[K]$ is a domain. Note that we have an obvious
$K$-grading on $\kk[K]$, and therefore a $G$-grading by taking the quotient. Let $A_g,g\in G$, be the components of this $G$-grading.
Note that $A_1=\kk[N/[N,N]]$, a ring of Laurent polynomials. Let $L$ be the fraction field of $A_1$. Then $D:=L\otimes_{A_1}A$
is the division algebra of quotients of $A$, and the $G$-grading on $A$ obviously extends to a (faithful) $G$-grading on $D$, as desired.
\end{remark}

\subsection{An example} Let us give an example of a non-abelian grading of a quaternion algebra ($d=2$).
\begin{example}
{\emph Let $D$ be the algebra over $\kk(a,b)$ generated by $x^{\pm 1}$ and $y$ with defining relations
$x+x^{-1}=a, \ y^2=b,$ and $yx=x^{-1}y.$ Set $z=x-x^{-1}.$ Then $D$ can also be generated over $\kk(a,b)$ by $y$ and $z$ with relations $y^2=b,\ z^2=a^2-4,$ and $yz=-zy$. We claim that this quaternion algebra does not split. Otherwise the equation
$(a^2-4)P^2+bQ^2=R^2$ should have nonzero solutions in polynomials of $a,b.$
We may assume that $P,Q,$ and $R$ have no common factors. We see that $P$ and $R$ must vanish when $b=0$ as $a^2-4$ is not a square. So $P=b\bar{P}, R=b\bar{R},$ and $(a^2-4)b\bar{P}^2+Q^2=b\bar{R}^2$. Hence $Q$ is divisible by $b$, which contradicts the assumption that $P,Q,$ and $R$ have no common factors. Therefore $D$ is a division algebra. \smallskip

Let us endow $D$ with a faithful grading by the quaternion group $Q_8$ as follows.
Set $c=a^2-2.$ Over the field $F:=\kk(c,b)$, the algebra $D$ has dimension $8$ and is generated by
$x^{\pm 1}$ and $y$ with relations $x^2+x^{-2}=c,\ y^2=b,$ and $yx=x^{-1}y$. Now put a grading on $F$ by $\{\pm 1\}$ setting $\deg(b)=\deg(c)=-1$. Even functions $(F_+)$ are in degree $1$ and odd functions $(F_-)$ in
degree $-1$. Extend this grading to a $Q_8$-grading on $D$ by setting $\deg(x)=i$ and $\deg(y)=j$. It is easy to check that this is indeed a grading. Namely:
$$\begin{array}{ll}
D[1]=F_+\oplus x^2F_-=:K, & \hspace{1cm} D[-1]=Kx^2=Kb=Kc, \\
D[i]=Kx, & \hspace{1cm} D[-i]=Kx^3, \\
D[j]=Ky, & \hspace{1cm} D[-j]=Kyb=Kyx^2, \\
D[k]=Kxy, & \hspace{1cm} D[-k]=Kxyb.
\end{array}$$
Note that $K$ is a non-central subfield of $D$, and the center $\kk(a,b)$ of $D$
is not graded ($a$ is not homogeneous).}
\end{example}

\section{Faithful group actions on division algebras}

\subsection{The result}
Question 3.5 in \cite{CEW} asks if any inner faithful action of a semisimple Hopf algebra on a division algebra $D$ must be Hopf-Galois. The main result of \cite{EW1} confirms this when $D$ is commutative: it shows that in this case the Hopf action must be a group action, so the result follows from classical Galois theory. The goal of this section is to prove the following theorem, which gives a positive answer to this question for group algebras.

\begin{theorem}\label{galo}
Let $D$ be a division algebra whose center contains $\kk$. Assume that a finite group $G$ acts on $D$.
Then $[D:D^G]$ divides $\vert G \vert$. Moreover, if the action is faithful, then the extension $D/D^G$ is $(\kk G)^*$-Galois.
\end{theorem}

\begin{remark} We will see below that the assumption that $\kk$ is algebraically closed is essential here: both Theorem \ref{galo} and the result of \cite{EW1} fail over algebraically non-closed fields.
\end{remark}

\begin{proof}
Set $Q=D^G$ and let ${\rm dim}_Q D$ denote the dimension of $D$ over $D^G$ as a left vector space. By a result of Cohen, Fischman, and Montgomery, see \cite[Theorem 8.3.7]{Mo}, it suffices to show that if the action of $G$ is faithful, then ${\rm dim}_Q D=|G|$. \smallskip

We have $D=\oplus_{V\in {\rm Irrep}(G)}V\otimes {\rm Hom}_G(V,D)$,
and ${\rm Hom}_G(V,D)$ are left (and right) vector spaces over $Q$. It is enough to prove that
$$
{\rm dim}_Q {\rm Hom}(V,D)=\dim V
$$
for each $V$. Then ${\rm dim}_Q D=|G|$. Let us denote this (left) dimension by $d_V$. By \cite[Corollary 2.3]{BCF}, $d_V<\infty$. Fix a nontrivial element $a\in G$ and let $C(a)$ be the cyclic subgroup of $G$ generated by $a$. Then $V=\oplus_{\chi\in C(a)^\vee} V_\chi\otimes \chi$, and $D=\oplus_{\chi\in C(a)^\vee}D_\chi$, where $D_\chi=\oplus_{V\in {\rm Irrep}(G)} V_\chi\otimes {\rm Hom}_G(V,D)$ is the subspace consisting of $x\in D$ such that $ax=\chi(a)x$. It is clear that all $D_\chi$ have the same dimension over $Q$, as right multiplication by any nonzero element of $D_\chi$ gives an isomorphism $D_1\cong D_\chi$ of left $Q$-vector spaces. So, $\sum_{V\in {\rm Irrep}(G)}{\rm dim}(V_\chi)d_V$ is independent on $\chi$. Multiplying this sum by $\chi(a)$ and summing over $\chi$, we thus get that $\sum_{V\in {\rm Irrep}(G)} {\rm tr}_V(a)d_V$ is zero (as the sum of all $n$-th roots of unity is zero for any $n>1$). Thus, the representation $\oplus_V d_VV$ is a multiple of the regular representation of $G$. But $d_{\kk}=1$ by definition, so this is exactly the regular representation, as desired.
\end{proof}

\subsection{A counterexample over the real field and in characteristic $p$.}

\begin{example} Example 4.6 of \cite{Mo1} shows that
Theorem \ref{galo} fails in positive characteristic.
The proof fails because the algebra $\kk [G]$ is not semisimple.
\end{example}

\begin{example}
Theorem \ref{galo} fails if the ground field $\kk$
is not algebraically closed. Indeed, take $D=\Bbb H$ (the algebra of quaternions over $\Bbb R$), and \linebreak $G\subset SU(2)\subset \Bbb H^\times$ the subgroup consisting of rotations of a platonic solid acting by conjugation. Then $|G|>[D:D^G]=4$, so the extension is not Hopf-Galois.
\end{example}

\begin{example}
The following example, which is a variation of \cite[Remark 4.3]{EW1},  shows
that if we work over $\Bbb R$, then a field may admit a coaction of a noncommutative but cocommutative
semisimple Hopf algebra, and such a coaction does not define a Hopf-Galois extension. \smallskip

Consider the group $G=S_3$ (the permutation group of 3 items), and let
$s_{12}, s_{23}\in G$ be the simple transpositions.
Consider the ${\Bbb C}$-antilinear automorphism $\tau: {\Bbb C} G\to {\Bbb C} G$
given by $\tau(s_{12})=s_{23}$, $\tau(s_{23})=s_{12}$, and let
$\mathcal{K}=({\Bbb C} G)^\tau$ be the space of its fixed points.
Then $\mathcal{K}$ is a 6-dimensional noncommutative
cocommutative $\Bbb R$-Hopf algebra, which is a form of ${\Bbb C}G$ over $\Bbb R$. \smallskip

The algebra $A={\Bbb C}[x,y]/(xy)$ can be inner faithfully graded by $G$ by setting ${\rm deg}(x)=s_{12}$ and ${\rm deg}(y)=s_{23}$.
Now let us pass to real forms. Let $x=u+iv$ and $y=u-iv$, then $xy=u^2+v^2$. Let $A_{\Bbb R}$ be the subalgebra of $A$
spanned by the polynomials of $u$ and $v$ with real coefficients; namely, $A_\Bbb R=A^\tau$, where
$\tau: A\to A$ is the antilinear involution defined by $\tau(x)=y$, $\tau(y)=x$.
Then $A_{\Bbb R}=\Bbb R[u,v]/(u^2+v^2)$, so it is an integral domain
(as $u^2+v^2$ is an irreducible polynomial over $\Bbb R$). \smallskip

We have
$$\begin{array}{l}
{\displaystyle \rho(u)= u \otimes \Big(\frac{s_{12}+s_{23}}{2}\Big)+v \otimes \Big(\frac{s_{23}-s_{12}}{2i}\Big),} \vspace{5pt}\\
{\displaystyle \rho(v)= u \otimes \Big(\frac{s_{12}-s_{23}}{2i}\Big) +v \otimes \Big(\frac{s_{12}+s_{23}}{2}\Big).}\vspace{5pt}
\end{array}$$
This means that $\rho$ restricts to a coaction $A_{\Bbb R}\to A_{\Bbb R}\otimes \mathcal{K}$.
This coaction can be extended to the field of quotients $L:=Q(A_{\Bbb R})$, by tensoring with $F:=\Bbb R(u^2)$ over $\Bbb R[u^2]$
(using that $u^2$ is a coinvariant of $\mathcal{K}$). We have $F=L^{\mathcal{K}}$. Thus we get a coaction of a noncommutative semisimple Hopf algebra $\mathcal{K}$ on a field
$L$, and $[L:F]=2$, while $|G|=6$, showing that this extension is not Hopf-Galois.
\end{example}

\section{The degree bound}

\subsection{The largest dimension of an irreducible representation and the PI degree}
Let $\kk$ be an algebraically closed field (of any characteristic), and $A$ an algebra over $\kk$. Let $F$ be an algebraically closed field containing $\kk$. A {\it matrix representation} of $A$ over $F$ is a $\kk$-linear homomorphism $\rho: A\to {\rm Mat}_n(F)$. Note that such a representation extends
naturally to $A\otimes F$. A representation $\rho$ is said to be irreducible if
the only $\rho(A)$-invariant subspaces of $F^n$ are $F^n$ and $0$. By the Density Theorem, this is equivalent to saying that $\rho(A)$ is a spanning set for ${\rm Mat}_n(F)$ over $F$, or $\rho(A\otimes F)={\rm Mat}_n(F)$. Let $d_*(A)$ be the largest dimension of an irreducible matrix representation of $A$. Also, let $d(A)$ be the PI degree of $A$, i.e., the smallest $n$ such that all polynomial identities of $n$ by $n$ matrices are satisfied in $A$. (We agree that $d_*(A)$ and $d(A)$ are equal to $\infty$ when they are not defined). Note that $d_*(A)\le d(A)$. \smallskip

Assume now that $A=\oplus_{s=1}^p A_s$, where $A_s$ is a finite dimensional algebra over a field $Z_s$ containing $\kk$. So $\oplus_s Z_s$ is contained in the center of $A$. In this case, by Schur's lemma, $Z_s$ acts by scalars in every irreducible matrix representation of $A$ (and for all but one $s$ they act by zero). Hence every irreducible matrix representation of $A$ has dimension less or equal than $r^{1/2}$ with $r={\rm max}_s\dim_{Z_s} A_s$. \smallskip

\begin{theorem}\label{theo}
Let $B$ be a $\kk$-subalgebra of $A$. Then $d_*(B)\le d_*(A)$.
\end{theorem}

\begin{proof}
Let $L_s$ be the algebraic closure of $Z_s$ and $\overline{A}:=\oplus_{s=1}^p A_s\otimes_{Z_s} L_s$. Consider $B$ as embedded in $\overline{A}$. Let $I$ denote the Jacobson radical of $\overline{A}$, \smallskip and $J=I\cap B$.

There exists $N$ such that $J^N=0$. So for any $x\in J$ and $a_0,\ldots,a_N\in \overline{A}$ we have $a_0xa_1x \ldots xa_N=0$. Let $\rho$ be an irreducible matrix representation of $B$ of dimension $n$ over a field $F$. Then for any $f_i\in F$, $b_i,b_i'\in B$, $i=1,\ldots,m$, the matrix $y=\sum_{i=1}^m f_i\rho(b_ixb_i')$ satisfies the equation $y^N=0$. By irreducibility of $\rho$, if $\rho(x)\ne 0$, then $y$ can be any element of ${\rm Mat}_n(F)$. Thus, $\rho(x)=0$, and so $\rho$ is pulled back from $B':=B/J$. \smallskip

We have that $B'$ is a subalgebra of $A':=\overline{A}/I$. The latter is a semisimple algebra, with matrix blocks of dimension less or equal than $d_*(A)$. Thus, the matrix algebra identities of degree $d_*(A)$ hold in $A'$, and hence in $B'$ and $B'\otimes F$. But $\rho: B'\otimes F\to {\rm Mat}_n(F)$ is surjective. Hence the matrix algebra identities of degree $d_*(A)$ are satisfied in ${\rm Mat}_n(F)$, which implies \linebreak $n\le d_*(A)$.
\end{proof}

This theorem can be generalized as follows.

\begin{corollary}\label{coro}
Let $R$ be a commutative $\kk$-algebra whose total quotient ring\footnote{The ring obtained by inverting all nonzero divisors of $R$.} $Z$ is a direct sum of finitely many fields, say $Z=\oplus_{i=1}^p Z_s$. Let $A$ be an $R$-algebra which is a finitely generated torsion-free module over $R$. For a $\kk$-subalgebra $B$ of $A$ we have $d_*(B)\le d_*(A)$.
\end{corollary}

\begin{proof}
Set $A_{\rm loc}:=A\otimes_R Z$. Since $A$ is torsion-free over $R$, we can view  $A$, and hence $B$, as a subalgebra of $A_{\rm loc}$. Each irreducible matrix representation of $A_{\rm loc}$ is an irreducible matrix representation of $A$, so its dimension is less or equal than $d_*(A)$. Thus the
result follows from Theorem \ref{theo} (indeed, $A_{\rm loc}=\oplus_{s=1}^p A_s$, where $A_s=e_sA$, and $e_s$ are the primitive idempotents of $Z$).
\end{proof}

\subsection{The result}
We keep the hypotheses of Corollary \ref{coro} on $R$ and $A$. Suppose that a finite dimensional Hopf algebra $\mathcal{K}$ over $\kk$ coacts on $A$ so that $A^{\mathcal{K}}\subset A$ is a Hopf-Galois extension. \smallskip

The goal of this section is to prove the following theorem.

\begin{theorem}\label{main}
One has $d(\mathcal{K})\le d(A\otimes A^{\rm op})$ and $d_*(\mathcal{K})\le d_*(A)^2$.
\end{theorem}

If $A=D$ is a central division algebra of degree $d$, then $d_*(A)=d$
and $d(A\otimes A^{\rm op})=d^2$, so Theorem
\ref{main} implies

\begin{corollary} If \hspace{1pt} $\mathcal{K}$ coacts on a central division
$\kk$-algebra of degree $d$ defining a Hopf-Galois extension, then the PI degree of $\mathcal{K}$ is at most $d^2$.
\end{corollary}

This gives a positive answer to \cite[Question 5.9]{EW1} for Hopf-Galois extensions. \smallskip

Note that it is shown in \cite{EW1} that this bound is sharp.

\subsection{Proof of Theorem \ref{main}}
We first show that $\mathcal{K} \subset B/I$, where $B$ is a subalgebra of $A\otimes A^{\rm op}$, and $I\subset B$ is an ideal. This will immediately imply the first statement. \smallskip

By hypothesis, the canonical map
$$can:A \otimes_{A^{co \mathcal{K}}} A \rightarrow A\otimes \mathcal{K}, \, a \otimes a' \mapsto aa'_{(0)} \otimes a'_{(1)}$$
is an isomorphism. It is easy to check that it is an homomorphism of \linebreak $A$-bimodules ($A$ acts on $A \otimes \mathcal{K}$ on the left via multiplication in the first component and on the right diagonally via the coaction). Identify $\End_{A-A}(A \otimes \mathcal{K})$ with $\End_{A-A}(A \otimes_{A^{co \mathcal{K}}} A).$ The map $\Phi:\mathcal{K} \rightarrow \End_{A-A}(A \otimes_{A^{co \mathcal{K}}} A), k \mapsto \Phi_k$, with $\Phi_k(a \otimes k')=a \otimes kk'$ for all $a \in A, k' \in \mathcal{K}$, is an injective algebra homomorphism. Let $I$ be the left ideal of $A \otimes A^{op}$ generated by the set \linebreak $\{a \otimes 1-1\otimes a : a \in A^{co \mathcal{K}}\}.$ Then $A \otimes_{A^{co \mathcal{K}}} A=(A \otimes A^{op})/I$ as $A$-bimodules. The right $A$-action on $A \otimes A^{op}$ is $(x \otimes y)\cdot c=x \otimes yc$. Notice that $I \cdot c \subset I$. This gives the right $A$-action on $(A \otimes A^{op})/I$. Consider now the $\kk$-subalgebra $B:=\{b \in A\otimes A^{op} : I b\subset I\}$. Clearly $I \subset B$ and $I$ is a 2-sided ideal of $B$. Moreover,
$\End_{A-A}((A \otimes A^{op})/I)=\End_{A \otimes A^{op}}((A \otimes A^{op})/I) \cong (B/I)^{op}$. Then we can view $\mathcal{K}^{op}$ as embedded in $B/I$, and hence $\mathcal{K}\subset B/I$ (as $\mathcal{K}\cong \mathcal{K}^{op}$ via the antipode). \smallskip

To prove the second statement, consider the algebra $\widetilde{\mathcal{K}}\subset B$, which is the preimage of $\mathcal{K}$ in $B/I$. We have $\widetilde{\mathcal{K}}\subset A\otimes A^{\rm op}$, and the algebra $A\otimes A^{\rm op}$ satisfies the conditions of Corollary \ref{coro} (as so does $A$). Also, $d_*(A\otimes A^{\rm op})=d_*(A)^2$ (as irreducible matrix representations of $A\otimes A^{\rm op}$ are $\rho_1\otimes \rho_2^*$, where $\rho_i$ are irreducible matrix representations of
$A$). So by Corollary \ref{coro}, we get that $d_*(\widetilde{\mathcal{K}})\le d_*(A)^2$. Hence $d_*(\mathcal{K})\le d_*(A)^2$.

\section{Actions of twisted group algebras on division algebras}

Let $G$ be a finite group, $A$ a finite abelian group with a $G$-action, and $\pi: G\to A$ a bijective 1-cocycle. Recall from \cite[Section 4]{EG} that we have a twist $J$ for the group $\Gamma:=A^\vee \rtimes G$ given by $J=\sum_{g\in G}g\otimes 1_{\pi(g)}.$

Our goal is to prove the following.

\begin{theorem}\label{mintri} The minimal triangular Hopf algebra $\mathcal{H}:=\kk[\Gamma]^J$ acts faithfully (hence inner faithfully) on a central division algebra of degree $|G|$.
\end{theorem}

\begin{proof}  Let $X$ be a finite set with a $G$-action and a $G$-equivariant map $\phi: X\to A$.
Then $X$ has the structure of a nondegenerate symmetric set; see \cite{ESS} (after Definition 2.3). Thus, we have a set-theoretical Yang-Baxter solution
$R: X^2\to X^2$. Consider the algebra $B$ with generators $b_x$, $x\in X$, and relations
$$
b_xb_y=b_{y'}b_{x'}
$$
with $R(x,y)=(x',y').$

\begin{lemma}\label{doma} $B$ is an integral domain.
\end{lemma}

\begin{proof} This follows from \cite[Corollary 1.6]{GV}.
Also, here is another proof. The algebra $B$ is contained in the group algebra ${\kk}[G_X]$ of the structure group $G_X$ of $X$ (namely, ${\kk}[G_X]$
is defined in the same way, but also adding the inverses of $b_x$). It is shown in
\cite{ESS} that the group $G_X$ is equipped with a bijective 1-cocycle
$\psi: G_X\to \Bbb Z^X$
and has a normal abelian subgroup of finite index, which is
free of finite rank. Also, it is shown in \cite{Ch} that the group $G_X$ is torsion-free.
Thus, it follows that $B$ is an integral domain, as the group algebra of a torsion-free group with an abelian subgroup of finite index is an integral domain; see \cite[Corollary 2 and Section 3]{Br}.
\end{proof}

Also, it is clear that $B={\kk}[b_X, x\in X]_J$, the twisted polynomial algebra.
(Note that $A^\vee \rtimes G$ acts on the ${\kk}$-span ${\kk}X$ of $X$ via
$gb_x=b_{gx}$, $g\in G$, and $fb_x=f(\phi(x))b_x$, $f\in A^\vee$).
Thus, the Hopf algebra $\mathcal{H}$ acts on $B$. Hence $\mathcal{H}$ acts on $Q_B$, the full localization of $B$, which is a division algebra by Lemma \ref{doma}.

Let us pick $X$ in such a way that the action of $G$ on $X$ is faithful, and $\phi(X)$ generates $A$.
For example, we may take $X=G\times S$, where $S$ is a set together with $\phi: S\to A$ such that $\phi(S)$ is a generating set for $A$ as a $G$-module,
define the action of $\Gamma$ on $X$ by $h(g,x)=(hg,x)$, and set $\phi(g,x)=g\phi(x)$.
Then the action of $A^\vee \rtimes G$ on ${\kk} X$ is faithful.
Thus, $\mathcal{H}$ acts inner faithfully on $B$ and $Q_B$, and $Q_B$ is a central division algebra of degree $|G|$, as desired.
\end{proof}

We see that as a $\Gamma$-module, ${\kk}X={\rm Ind}_{A^\vee}^{A^\vee \rtimes G}{\rm Fun}(S,\kk)$. We can take $S=A$, $\phi|_S={\rm Id}$,
then ${\rm Fun}(S)={\kk}A^\vee$ (via the Fourier map), so $X=\Gamma$ and ${\kk}X$ is ${\kk}\Gamma$, the regular representation of $\Gamma$. Similarly, we can take $X=\Gamma\times
\lbrace{1,\ldots,m\rbrace}$,
so ${\kk}X =({\kk}\Gamma)^{\oplus m}$. This gives rise to the following
generalization.

\begin{corollary} Let $K$ be any finite group containing $\Gamma$. Then ${\kk}[K]^J$ acts on a central division algebra of degree $|G|$ so that the action is Hopf-Galois.
\end{corollary}

\begin{proof} As a $\Gamma$-module, ${\kk}[K]$ is a multiple of the regular representation. So setting ${\mathcal{R}}=\kk[b_x, x\in K]$,
with $gb_x=b_{gx}, g\in K$, and taking $B={\mathcal{R}}_J$, we get an action of $\kk [K]^J$ on $Q_B$, which is as required. The fact that the action is Hopf-Galois is clear from ordinary Galois theory, since the invariants and their action do not change under twisting.
\end{proof}
\smallskip

\begin{remark}\label{zhangrem}
After this paper was finished, James Zhang explained to us that if $\Gamma$ is a finite group and $J$ is a twist in $\kk[\Gamma]^{\otimes 2}$, then $\kk[\Gamma]^J$ acts inner faithfully on a central division algebra,
of degree $d:=|H|^{1/2}$, where $H\subset \Gamma$ is the support subgroup of $J$ (i.e., a subgroup of minimal order such that $J$ is conjugate to a twist in $\kk[H]^{\otimes 2}$). This is a generalization of Theorem \ref{mintri}. Zhang's proof is as follows. \par \smallskip

Let $V$ be a faithful finite dimensional representation of $\Gamma$. Then $\Gamma$ acts faithfully on the symmetric algebra $SV$. So $\kk[\Gamma]^J$ acts on $(SV)_J$, the algebra
$SV$ with multiplication twisted by $J$. First, we show that the global dimension of $(SV)_J$ equals $\dim V$. Indeed, it is well known that the smash product algebra $SV\#\kk[\Gamma]$ is isomorphic to $(SV)_J\# \kk[\Gamma]^J$, so we have
$$
{\rm gldim} \, (SV)_J={\rm gldim}\,  (SV)_J\# \kk[\Gamma]^J={\rm gldim}\,  SV\# \kk[\Gamma]= \dim V,
$$
where the first equality holds, for example, by \cite[Theorem 1.1]{Li}. Also, $(SV)_J$ is a PI algebra, since $(SV)^\Gamma$ is contained in the center of $(SV)_J$, and by the Hilbert-Noether Theorem $(SV)_J$ is a finite module over its center. Hence $(SV)_J$ is Noetherian. But a Noetherian PI algebra is FBN by \cite[Corollary 13.6.6(iii)]{MR}. Thus, $(SV)_J$ is a connected $\Bbb Z_+$-graded FBN algebra of finite global dimension. By \cite[Corollary 1.2]{SZ}, any such algebra must be a domain. Hence $(SV)_J$ is a domain. Using \cite[Theorem 2.2]{SV}, we can extend the action of $\kk[\Gamma]^J$ to the central division algebra $Q$ of quotients of $(SV)_J$. Moreover, this action is faithful, hence inner faithful. Also, it is easy to show that the degree of $Q$ is $d$, as desired.
\end{remark}
\smallskip

\subsection*{Acknowledgments}
The authors are grateful to E. Aljadeff and C. Walton for many useful discussions, and to K. Brown and J. Zhang for providing the arguments in Remarks \ref{brownrem} and \ref{zhangrem}. \par \smallskip

The work of P. Etingof was partially supported by the NSF grant DMS-1000113. J. Cuadra was supported by grant MTM2014-54439 from MINECO and FEDER and by the research group FQM0211 from Junta de Andaluc\'{\i}a. This work was mainly done during his visit to the Mathematics department of MIT. He is deeply grateful for the generous hospitality and inspiring atmosphere. This visit was financed through grant PRX14/00283 from the Spanish Programme of Mobility of Researchers.

\end{document}